\documentclass[a4paper,oneside,12pt]{article}

\usepackage{amsmath,amsfonts,amscd,amssymb,amsthm}
\usepackage{longtable,geometry}
\usepackage[english]{babel}
\usepackage[utf8]{inputenc}
\usepackage[active]{srcltx}
\usepackage[T1]{fontenc}
\usepackage{graphicx}
\usepackage{bbm}

\usepackage{color}
\usepackage{MnSymbol}
\usepackage{nicefrac}
\usepackage{calrsfs}
\usepackage[pdfborder={0 0 0}, pdfborderstyle={},colorlinks=true,linkcolor=black,citecolor=black,urlcolor=blue]{hyperref}

\newtheorem{theorem}{Theorem}

\newtheorem{lemma}[theorem]{Lemma}

\theoremstyle{remark}

\newcommand{\be}[1]{\begin{equation}\label{#1}}
\newcommand{\ee}{\end{equation}}
\newcommand{\ba}[1]{\begin{align}\label{#1}}
\newcommand{\ea}{\end{align}}
\newcommand{\ben}{\begin{equation*}}
\newcommand{\een}{\end{equation*}}
\newcommand{\calA}{\mathcal{A}}
\newcommand{\calB}{\mathcal{B}}
\newcommand{\calC}{\mathcal{C}}

\newcommand{\bbN}{\mathbb{N}}

\newcommand{\bbP}{\mathbb{P}}

\newcommand{\bbR}{\mathbb{R}}

\newcommand{\bbZ}{\mathbb{Z}}
\newcommand{\ep}{\varepsilon}

\title{Seven-dimensional forest fires}
\author{Daniel Ahlberg, Hugo Duminil-Copin, Gady Kozma\\ and Vladas Sidoravicius}
\date{\today}

\begin{document}
% Entete %
\maketitle

\begin{abstract}
We show that in high dimensional Bernoulli bond percolation, removing from a thin
infinite cluster a much thinner infinite cluster leaves an
infinite component. This observation has implications for the van den Berg-Brouwer
forest fire process, also known as self-destructive percolation, for dimension high enough.

%\medskip
%{\centering\bf R\'esum\'e \par}
%\medskip
%Cette article montre que dans la percolation de Bernoulli par ar\^ete en grande dimension, retirer d'une composante connexe infinie de faible densit\'e une composante connexe de densit\'e beaucoup plus faible laisse une composante connexe infinie. Cette observation a des implications pour le processus de feux de for\^et de van den Berg-Brouwer, \'egalement connu sous le nom de percolation auto-destructive, en dimension suffisamment grande.

%\begin{keywords}
%Near-critical percolation, static renormalization.
%\end{keywords}
%\begin{msc}
%60K35, 82B43.
%\end{msc}
\end{abstract}

\section{Introduction}

Think about the open vertices of supercritical site percolation as if they were
trees, and about the infinite cluster as a forest. Suddenly a fire
breaks out and the entire forest is cleared. New trees then start
growing randomly. When can one expect a new infinite cluster to
appear? The surprising conjecture in \cite{BB04} is that in the
two-dimensional case, even if the original forest were extremely thin,
still a considerable amount of trees must be added to create a new
infinite cluster. Heuristically, the conjecture claims that the
infinite cluster might occupy a very low proportion of vertices but they sit in a way
that separates the remaining finite clusters by gaps that cannot be
easily bridged. This conjecture is still open. See
\cite{BB04,BBV08,BL09} for connections to other models of forest fires
and more.

Let us define the model formally, in three steps. The model was originally introduced as a site percolation model, but we will define it for bonds, as some of the auxiliary results we need have only been
proved for bond percolation. We are given a
graph $G$, a probability $p\in [0,1]$ (``the original density'') and a
probability $\ep\in[0,1]$ (``the recovered density''). Let $\bbP_p$ be the Bernoulli bond percolation measure on $G$ with parameter $p$.
\begin{enumerate}
\item\label{pg:def}
%Let $\bbP_p$ be the Bernoulli bond percolation measure on $G$ with edge-parameter $p$. The configuration is denoted $\omega_p\in\{0,1\}^{E(G)}$. A {\em cluster} is a maximal connected component of edges.
Assign independent uniformly distributed values from $[0,1]$ to the edges of $G$. Let $\omega_p\in\{0,1\}^{E(G)}$ denote the set of edges with value at most $p$. The configuration $\omega_p$ is distributed as $\bbP_p$, and a {\em cluster} refers to a maximal connected component of edges. It will be of importance below that as $p$ ranges over $[0,1]$, we obtain a simultaneous coupling of Bernoulli configurations on $G$ such that $\omega_{p_1}\subset\omega_{p_2}$ when $p_1\le p_2$.
\item Let $\tilde \bbP_p$ be the law of the configuration $\tilde \omega_p$ constructed as follows: for any edge $e$, 
$$\tilde\omega_p(e)=\begin{cases}\omega_p(e)&\text{ if $e$ is in a
  finite cluster of $\omega_p$,}\\ 0&\text{ otherwise}.\end{cases}$$ 
\item Let $\tilde\bbP_{p,\ep}$ be the law of $\tilde \omega_{p,\ep}$ where $\tilde\omega_{p,\ep}$ is defined as follows: for any edge $e$, $\tilde \omega_{p,\ep}(e)=\max\{\tilde \omega_p(e),\omega'_\ep(e)\}$, where $\omega'_\ep$ is a percolation configuration with edge-weight $\ep$, which is independent of $\omega_p$. 
\end{enumerate}
We can now define our property of interest.

\paragraph{Definition.} Let $p_c(G)$ denote the critical threshold for bond percolation on a graph $G$. We say that $G$ \emph{recovers from fires} if
  for every $\ep>0$, there exists $p>p_c(G)$ such that
  $\tilde\bbP_{p,\ep}$ has an infinite connected component, with
  probability 1. 
We say that $G$ \emph{site-recovers from fires} if the analogous
definitions for site percolation hold.

\medskip
%$\tilde\bbP_{p,\ep}$ is also known as the self-destructive percolation process. 

In \cite{BB04} the authors showed that a binary tree site-recovers from
fires and conjectured that $\bbZ^2$ lattice does \emph{not} site-recover from
fires. The binary tree is an example of a non-amenable graph, that is, a graph in which the boundary of a (finite) set of vertices is comparable in size to the set itself. Recovery from fires, both in edge and site sense, was proven in \cite{AST12} for a large class of non-amenable transitive graphs. Our result concerns hyper-cubic lattices.

\begin{theorem}\label{main}For $d$ sufficiently large, $\bbZ^d$
  recovers from fires.
\end{theorem}

Here and below, $\bbZ^d$ refers to the $\bbZ^d$ nearest neighbour lattice. The main property of $\bbZ^d$ that we will use is that
$\bbP_{p_c}(0\longleftrightarrow\partial B(0,r))\le Cr^{-2}$ (see below for a
discussion on this condition, and also for the notations). This was
proved in \cite{KN11} based on results of Hara, van der Hofstad \&
Slade \cite{HHS03,H08}. These establish the necessary estimate for $d$
sufficiently large (19 seems to be enough, though this can be
improved) and also for \emph{stretched-out} lattices in $d>6$. The
number $6$ is actually meaningful and is the limit of the technique
involved, lace expansion. Our proof easily extends to the stretched-out
7-dimensional lattice (hence the title of the article), but for
simplicity we will prove the theorem only for nearest-neighbour
percolation in $d$ sufficiently high. In fact, our proof provides further information in the supercritical percolation regime.
%Consider the following construction of $\omega_p\in\{0,1\}^{E(G)}$. Assign to the edges of $\bbZ^d$ independent uniformly distributed random variables on $[0,1]$, and let $\omega_p$ denote the indicator of the set of edges with value at most $p$. The law of $\omega_p$ equals $\bbP_p$, so $\{\omega_p\}_{p\in[0,1]}$ forms a coupling of increasing percolation configurations on $\bbZ^d$. Denote this coupling by $\mathbf P$. Moreover, let $\calC_\infty(\omega_p)$ denote the infinite cluster present in $\omega_p$.
Recall the common notation $\calC_\infty(\omega_p)$ for the infinite cluster of edges present in $\omega_p$.

\begin{theorem}\label{thm:thinthinner}For every $\ep>0$ and $d$ sufficiently large, there
  exists $p>p_c$ such that
  $\omega_{p_c+\ep}\setminus\calC_\infty(\omega_p)$ contains an
  infinite cluster almost surely.
\end{theorem}
\bigskip

Theorem \ref{main} is clearly a corollary of Theorem
\ref{thm:thinthinner}. Another consequence is that for every $\ep>0$, the critical
probability for percolation on the random graph obtained from $\bbZ^d$
by removing a sufficiently `thin' supercritical percolation cluster, that is $\calC_\infty(\omega_{p_c+\delta})$ for small enough $\delta=\delta(\ep)>0$, is almost
surely at most $p_c+\ep$. Theorem \ref{thm:thinthinner} and the last statement cannot possibly hold for site percolation on $\bbZ^2$, since an infinite cluster cuts space up into finite pieces.

%Van den Berg conjecture asserts that there exists $\ep_0>0$ such that for any $\delta,\ep<\ep_0$, $\tilde\bbP_{p_c+\delta,\ep}(0\longleftrightarrow \infty)=0$.

%Let $\theta(p,\ep)=\tilde\bbP_{p,\ep}(0\longleftrightarrow\infty)$ and set $\ep_c(p)=\inf\{\ep>0:\theta(p,\ep)>0\}$. 

%The model can be defined on more general graphs and in particular on so-called spread-out percolation models. The theorem applies for sufficiently spread-out models in dimension 7 or higher for the following reason. The main ingredient of the proof is the fact that the so-called one-arm exponent is larger than 1. This exponent has been computed in \cite{KN11} under mild assumption which are verified as soon as lace-expansion techniques are available.
 
\paragraph{Proof sketch.} We will show that for every $\ep>0$, there
exists some $p>p_c$ such that when removing the infinite cluster of
$p$-percolation from $(p_c+\ep)$-percolation, the remainder still
percolates. The proof proceeds by a renormalization procedure. 
\begin{enumerate}
\item We first choose $\ell\in \bbN$ sufficiently
large such that for any $L\ge \ell$, connectivity properties of boxes of size $L^2\times\ell^{d-2}$ in $(p_c+\ep)$-percolation behave  like
$(1-\eta)$-percolation on a coarse grain lattice for some small $\eta$. This is a standard application of
Grimmett-Marstrand \cite{grimar90} and renormalization theory. 
\item We then use the fact
that the one-arm exponent in high dimensions is 2 to note that for any
$L$, only a small number $M$ of vertices in a box of size $L^2\times\ell^{d-2}$ 
can connect to distance $L$ in \emph{critical percolation}.  
\item Picking
$L$ sufficiently large, one can argue that these $M$ points do not alter the connectivity properties of boxes of size $L^2\times\ell^{d-2}$ for $(p_c+\ep)$-percolation. In particular, the coarse grain percolation still behaves like
$(1-\eta)$-percolation even after removing that small number of
vertices. 
\item We now pick $p$ sufficiently close to $p_c$ that the behaviour (for $\omega_p$) at
scale $L$ is not altered by moving from $p_c$ to $p$. Since there are less sites in $\calC_\infty(\omega_p)$ than sites connected to distance $L$ in $\omega_p$, this $p$ gives the result.
\end{enumerate}
Examining this a little shows that what the proof really needs is that
the one-arm exponent is bigger than 1, i.e.\ that
\[
\bbP_{p_c}(0\longleftrightarrow\partial B(0,r))\le r^{-1-c}\qquad c>0.
\]
The number of points removed in the second renormalization step will
in this case no longer be bounded independently of $L$, but would still be too
small to block the cluster of the boxes at scale $\ell$.
This is interesting as it is conjectured to hold also below 6
dimensions. While nothing is proved, simulations hint that it might
hold for $\bbZ^5$ \cite[\S 2.7]{AS94}. On the other hand, let us note that in $\bbZ^3$ this probability is larger than $ cr^{-1}$ (this is well-known but we
are not aware of a precise reference -- compare to
\cite[(3.15)]{vdBK85} and \cite[Theorem 5.1]{Kes82}). Hence, the
approach used here has no hope of working in $\bbZ^3$ (though, of course,
this does not preclude the possibility that $\bbZ^3$ does recover from
fires). We remark that a similar renormalization technique was
recently used in \cite{GHK}, also under the assumption that the
one-arm exponent is bigger than 1.

\paragraph{Notations.} Identify $\bbZ^2$ with the subgraph of $\bbZ^d$
of points with the $d-2$ last coordinates equal to 0. Let
$S_\ell=\{x\in\bbZ^d:|x_i|\le \ell\;\forall i\ge 3\}$ be the two-dimensional slab of
height $2\ell+1$. Recall also the following standard notations: 
For a subgraph $G$ of $\bbZ^d$, we say that $x$ is connected to $y$ in $G$ if they are in the same connected component of $G$. We denote this by $x\stackrel{G}{\longleftrightarrow}y$ or simply $x\longleftrightarrow y$ when the context is clear. We will use the notation $x\longleftrightarrow A$ to denote the fact that $x\longleftrightarrow y$ for some $y$ in $A\subset\bbZ^d$.

Let $||\cdot||_\infty$ be the infinity norm on $\bbR^d$ defined by
$$||x||_\infty=\max\{|x_i|:i=1,\dots,d\}.$$ 
We consider the hypercubic
lattice $\bbZ^d$ for some large but fixed $d$. For $\ell,L>0$, define the ball
$B_x(L)=\{y\in\bbZ^d:||y-x||_\infty\le L\}$ and let $\partial B_x(L)$
be its inner vertex boundary. %Define $\bbL_L=2L\bbZ^2$ the rescaled version of $\bbZ^2$ naturally embedded in $\bbZ^d$. 

\paragraph{Acknowledgements.} 
During this work Daniel Ahlberg was supported by the Brazilian CNPq through postdoctoral fellowship 150804/2012-1. Gady Kozma's work was partially supported by the Israel Science Foundation. Hugo Duminil-Copin was supported by the NCCR SwissMap, the ERC AG CONFRA, as well as the Swiss FNS. Vladas Sidoravicius thanks Weizmann Institute and the Forschungsinstitut f\"ur Mathematik at ETH, for their hospitality and financial support. The research of Vladas Sidoravicius was supported in part by Brazilian CNPq Grants 308787/2011-0 and 476756/2012-0 and FAPERJ Grant E-26/102.878/2012-BBP. This work was also supported by ESF RGLIS Excellence Network.

\section{Proof}
From now on, $d$ is fixed and large enough. For $x\in \bbZ^2$, let $\calA(x,\ell,L,M)$ be the event that there are
less than $M$ sites $y$ in the $(6L+1)\times (6L+1)\times (2\ell+1)^{d-2}$ box $S_\ell\cap B_x(3L)$ that are connected to a site at distance $L$ from themselves.
%in a cluster of radius larger than $L$.
Note that we do not assume that this connection is contained in the slab $S_\ell$, the connection may be anywhere in $B_y(L)$. 

\begin{lemma}\label{lem:1}
Let $\eta>0$ and $\ell>0$. There exists $M>0$ such that for any integer $L$, there exists $p>p_c$ such that
$$\bbP_{p}(\calA(x,\ell,L,M))\ge 1-\eta.$$
\end{lemma}

\begin{proof}
By \cite{KN11}, there exists $C>0$ such that (for large enough $d$)
\begin{equation}\label{eq:1}\bbP_{p_c}(0\longleftrightarrow \partial B_0(n))\le \frac C{n^2}.\end{equation}
Choose $M$ in such a way that $\frac {49(2\ell+1)^{d-2} C}M< \eta$. For any integer $L$, Markov's inequality implies
\begin{multline*}\bbP_{p_c}\big[|\{y\in S_\ell\cap
    B_x(3L):y\leftrightarrow \partial B_y(L)\}|\ge M \big]\\
\le
  \frac{1}{M}\sum_{y\in S_\ell\cap B_x(3L)}
  \bbP_{p_c}(y\leftrightarrow \partial B_y(L)).
\end{multline*}
By \eqref{eq:1} and the choice of $M$, the right-hand side is thus
strictly smaller than $\eta$. By choosing $p$ close enough to $p_c$, we obtain that 
\[\bbP_{p}\big[|\{y\in S_\ell\cap
    B_x(3L):y\longleftrightarrow \partial B_y(L)\}|\ge M \big]\le
  \eta.\qedhere
\]
\end{proof}

For a set $S\subset\bbZ^d$, let $\omega^S$ be the configuration obtained from $\omega$ by closing each edge adjacent to some site in $S$. Let $\calB(x,\ell,L,M)$ be the event that for any set $S$ of $M$ sites contained in $B_x(3L)$, $\omega^S$ contains
\begin{itemize}
\item a cluster crossing from $\partial B_x(L)$ to $\partial B_x(3L)$
  contained in the slab $S_\ell$,
\item a unique cluster in the box $S_\ell\cap B_x(3L)$ of radius larger than $L$.
\end{itemize}

\begin{lemma}\label{lem:2}
Let $\eta>0$ and $\ep>0$. There exists $\ell>0$ such that for any $M>0$, there is $L>0$ so that $$\bbP_{p_c+\ep}(\calB(x,\ell,L,M))\ge 1-\eta.$$
\end{lemma}
%The integer $\ell$ is of the order of the correlation length of percolation with parameter $p_c+\ep$. s
\begin{proof}
For a given $\ell$ and $L$ denote by $E=E(x,\ell,L)$ the event that:
\begin{enumerate}
\item There is a crossing from $\partial B_x(L)$
  to $\partial B_x(3L)$ in $S_\ell$.
\item There is exactly one cluster in $S_\ell\cap
  B_x(3L)$ of radius larger than $L$.
\end{enumerate}
Shortly, the event $E$ is just $\calB$ without the set $S$, or, if you
want, $\calB$ is the event that $E$ occurred in $\omega^S$ for
all $S$ with $|S|\le M$. 

We claim that for $\ell$ sufficiently large, $\bbP_{p_c+\ep}(\neg E)\le \exp(-cL)$ for
some $c=c(\ep,\ell)>0$ independent of $L$. Finding
such an $\ell$ is a standard exercise in renormalization theory, but
let us give a few details nonetheless.
Call a box of side-length $2\ell+1$ {\em good} if it contains crossings between opposite faces in all directions, and if all clusters of diameter at least $\frac 14 \ell$ connect inside the box. By choosing $\ell$ large, we can require that a box is good with arbitrarily high probability (see e.g.\ the appendix of \cite{BBHK08}).
Considering such boxes centered around the sites in $\ell\bbZ^2$. The
events that these boxes are good are 2-dependent (in the sense of
\cite{LSS97}, i.e.\
%any box is independent of all boxes not neighbouring it
disjoint boxes are good independently), and
hence by \cite{LSS97}, if the probability that a box is good is sufficiently
large, then the good boxes stochastically dominate two-dimensional site percolation at density, say, $\frac 9{10}$. Now, a cluster of good boxes contains
a cluster in the underlying percolation, 
since the crossings of adjacent boxes must intersect. This means
that if either of the conditions in the definition of $E$ fail, then
there is an $\ell_\infty$-cluster of bad boxes containing at least $L/\ell$ boxes. (Here an $\ell_\infty$-cluster refers to a maximal set of connected sites with respect to $\ell_\infty$-distance, as opposed to $\ell_1$-distance used elsewhere.) But the
probability for that, from Peierls' argument, is at most $(8/10)^{L/\ell}\cdot
(6L/\ell)^2$. This shows the claim.

Fix $M>0$. Let $F_M$ be the set of configurations in $B_x(3L)$ for which there exists $S\subset B_x(3L)$ with $|S|=M$ and $\omega^S\not\in E$. We have
 \begin{align*}\bbP_{p_c+\ep}(F_M)&\le \sum_{S\subset B_x(3L):\,|S|=M}\bbP_{p_c+\ep}(\omega^S\not\in E)\\
 &\le \sum_{S\subset
     B_x(3L):\,|S|=M}(1-p_c-\ep)^{-2dM}\,\bbP_{p_c+\ep}(\neg E)\\
 &\le (1-p_c-\ep)^{-2dM}(6L+1)^{dM}\,\bbP_{p_c+\ep}(\neg E)\\
 &\le(1-p_c-\ep)^{-2dM}(6L+1)^{dM}\exp(-cL).\end{align*}
 For $L$ large enough, this quantity is smaller than $\eta$. The lemma follows from the fact that if $\omega\notin\calB(x,\ell,L,M)$, then there exists $S\subset B_x(3L)$ with $|S|=M$ and $\omega^S\not\in E$, i.e.\ $\omega\in F_M$.
\end{proof}

In order to prove Theorem~\ref{main} and~\ref{thm:thinthinner}, we will use Lemma~\ref{lem:2} to construct an infinite cluster at density $p_c+\ep$, and Lemma~\ref{lem:1} to make sure that the infinite cluster present at the lower density $p$ does not interfere too much with this construction.

\begin{proof}[Proof of Theorems \ref{main} and \ref{thm:thinthinner}]Recall the notations $\omega_p$,
  $\tilde\omega_p$ and $\omega'_\ep$ from page \pageref{pg:def}. We need to show that for any $\ep>0$,
  there exists $p>p_c$ such that $\tilde\omega_{p,\ep}$ has an
  infinite component. Note that $(\omega_{p_c}\cup\omega'_\ep)\setminus \calC_\infty(\omega_{p})$ is stochastically dominated by $\tilde \omega_{p,\ep}$. Thus, it suffices to show that for every $\ep>0$, there is $p>p_c$ such that $\omega_{p_c+\ep}\setminus \calC_\infty(\omega_{p})$ contains an infinite component. That is, Theorem~\ref{main} follows from Theorem~\ref{thm:thinthinner}, and it suffices to prove the latter.
  
  Let therefore $\ep>0$.
Fix $\eta>0$ such that $1-2\eta$ exceeds the critical parameter for any $8$-dependent percolation on vertices of $\bbZ^2$. Define successively $\ell,M,L$ and $p$ as follows. Fix $\ell=\ell(\ep,\eta)>0$ as defined in Lemma~\ref{lem:2}. Pick $M=M(\eta,\ell)>0$ as defined in Lemma~\ref{lem:1}. This defines $L=L(\eta,\ep,\ell,M)>0$ by Lemma~\ref{lem:2}, and then $p=p(\eta,\ell,M,L)>p_c$ by Lemma~\ref{lem:1}.

Let $\mathbf P$ denote the joint law of $(\omega_p,\omega_{p_c+\ep})$ under the increasing coupling described above. A site $x\in L\bbZ^2$ is said to be {\em good} if $\omega_p\in\calA(x,\ell,L,M)$ and $\omega_{p_c+\ep}\in\calB(x,\ell,L,M)$. By definition, 
$$\mathbf P\big[\calA(x,\ell,L,M)\cap\calB(x,\ell,L,M)\big]\ge 1-2\eta.$$
Since these events depend on edges in $B_x(4L)$ only, the site percolation (on $L\bbZ^2$) thus obtained is 8-dependent. As a consequence, there exists an infinite cluster of good sites on the coarse grained lattice $L\bbZ^2$. 

On the event that there exists an infinite cluster of good sites on the coarse grained lattice, there exists an infinite path in 
$\omega_{p_c+\ep}\setminus \calC_\infty(\omega_{p})$.
Indeed, by induction, consider a path of adjacent good sites $x_1,\dots,x_n$. Consider $C_i$ to be a cluster in 
$$[\omega_{p_c+\ep}\setminus
  \calC_\infty(\omega_{p})]\cap[B_{x_i}(3L)\setminus B_{x_i}(L)]$$ of
radius larger than $L$. By the definition of $\calA$ there are at most
$M$ sites in the box $S_l\cap B_{x_i}(3L)$ connected to distance
$L$ in $\omega_p$. Hence the same box also contains no more than $M$ sites in
$\calC_\infty(\omega_p)$ since any site connected to infinity must be
connected to distance $L$. Using the definition of $\calB$ with $S$
being exactly $\calC_\infty(\omega_p)\cap S_l\cap B_{x_i}(3L)$ we see
that $\omega_{p_c+\ep}\setminus\calC_\infty(\omega_p)$ contains a
crossing cluster for the box $S_l\cap B_{x_i}(3L)$ with all the
properties listed before Lemma \ref{lem:2}. In particular, the
uniqueness property ensures two such crossing clusters in two
neighbouring boxes must intersect. The result follows readily.
\end{proof}

%\bibliographystyle{amsalpha}
%\bibliography{sdp-7D}

\newcommand{\noopsort}[1]{}\def\cprime{$'$}
\providecommand{\bysame}{\leavevmode\hbox to3em{\hrulefill}\thinspace}
\providecommand{\MR}{\relax\ifhmode\unskip\space\fi MR }
\providecommand{\MRhref}[2]{%
  \href{http://www.ams.org/mathscinet-getitem?mr=#1}{#2}
}
\providecommand{\href}[2]{#2}

\begin{flushright}
\footnotesize\obeylines
  \textsc{IMPA}
  \textsc{Rio de Janeiro, Brazil}
  \textsc{E-mail:} \texttt{ahlberg@impa.br}
 \medbreak
  \textsc{Universit\'e de Gen\`eve}
  \textsc{Gen\`eve, Switzerland}
  \textsc{E-mail:} \texttt{hugo.duminil@unige.ch}
 \medbreak
    \textsc{Weizmann Institute}
  \textsc{Rehovot, Israel}
  \textsc{E-mail:} \texttt{gady.kozma@weizmann.ac.il}
\medbreak
    \textsc{IMPA}
  \textsc{Rio de Janeiro, Brazil}
  \textsc{E-mail:} \texttt{vladas@impa.br}
\end{flushright}
%\begin{flushright}
%\footnotesize\obeylines
%  \textsc{Universit\'e de Gen\`eve}
%  \textsc{Gen\`eve, Switzerland}
%  \textsc{E-mail:} \texttt{hugo.duminil@unige.ch}
%  \end{flushright}
 
\end{document}